\documentclass[hidelinks,11pt]{article}
\usepackage{amsmath,amssymb,latexsym,mathdots,comment,enumerate}
\usepackage{parskip}
\usepackage{commath}
\usepackage{amsthm}
\usepackage{graphicx}
\usepackage[all,cmtip]{xy}
\usepackage{color}

\usepackage[utf8]{inputenc}
\usepackage[T1]{fontenc}
\usepackage{lmodern}
\usepackage{mathtools, nccmath, textcomp} 
\usepackage{bigints}
\usepackage{hyperref}
\usepackage[italicdiff]{physics}

\DeclareMathSymbol{*}{\mathbin}{symbols}{"03}

\def\R{{\mathbb R}}

\def\LL1{{\L(1)}}

\def\L{{\mathcal L}}

\def\max{{\rm max}}

\newtheorem*{thr*}{Theorem}
\newtheorem{theorem}{Theorem}[section]

\newtheorem*{lm*}{Lemma}

\newtheorem{df}{Definition}

\begin{document}
\title{On the Density of naturals $n$ coprime to $\lfloor P(n) \rfloor$ for certain Classes of Polynomials}
\author{Aahan Chatterjee}

\maketitle

\vspace{-.2in}
\begin{abstract}
We obtain asymptotic bounds on the number of natural numbers less than $X$ satisfying $\gcd{\left(n,\lfloor P(n) \rfloor\right)}=1$, under some diophantine conditions on the coefficient of $x$ in $P$, and show that the density of such naturals is exactly $\dfrac{1}{\zeta{\left(2\right)}}$.
\end{abstract}

%\begin{itemize}
    %\item Give history;
    %\item State main theorems;
    %\item State applications;
    %\item Briefly describe the method of proof;
    %\item Describe the organization of the paper.
%\end{itemize}}

\section{Introduction}
If \( m \) and \( n \) are two natural numbers chosen at random, the probability that they are coprime is known to be \( \frac{6}{\pi^2} \). This result extends beyond independent integers and can hold when \( m \) and \( n \) are functionally related. For instance, Watson proved in \cite{Watson} that if \( \alpha \) is irrational, the density of natural numbers \( n \) satisfying \( \gcd(n, \lfloor \alpha n \rfloor) = 1 \) is precisely \( \frac{6}{\pi^2} \). Similarly, Lambek and Moser demonstrated in \cite{Lambek_Moser} that if \( f(1), f(2), \dots \) is a non-decreasing sequence of non-negative integers that grows slowly to infinity, and if the intervals over which \( f(m) = n \) grow slowly with \( n \), then the probability that \( f(n) \) is relatively prime to \( n \) is also \( \frac{6}{\pi^2} \).

In this paper, we extend Watson’s result to a broader class of functional relationships. Roughly speaking, we show that if \( P \) is a polynomial with real coefficients and the coefficient of \( x \) in \( P \) is not too well-approximable, then the probability that \( n \) and \( \lfloor P(n) \rfloor \) are coprime is exactly \( \frac{6}{\pi^2} \).

\subsection*{Main Result}
More specifically:
\begin{thr*}
Let $P \in \mathbb{R}[x]$ be a polynomial such that the coefficient of $x$ in $P$ is \textbf{non-Liouville} (we shall define this in \textbf{Definition \ref{df}}). Then we have that:
\begin{equation*}
\begin{split}
S(X)&=\dfrac{X}{\zeta{\left(2\right)}}+O\left(\dfrac{X}{\log{\log{X}}^{1\slash 3}}\right)\quad \text{and hence}\quad \lim_{X \rightarrow \infty}\dfrac{S(X)}{X}=\frac{1}{\zeta(2)}\\ 
\text{where} \quad S(X)&=: \{x \le X: \gcd{\left(x,\lfloor P(x) \rfloor\right)}=1\}.
\end{split}    
\end{equation*}
\end{thr*}

\subsection*{Heuristic}
For a prime \( p \), the probability that \( p \mid \gcd{\left(n, \lfloor P(n)\rfloor\right)} \) is equal to \( \frac{1}{p} \) times the probability that \( \left\{\frac{P(pn)}{p}\right\} < \frac{1}{p} \), which is \( \frac{1}{p^{2}} \). Therefore, the probability that \( p \) does not divide \( \gcd{\left(n, \lfloor P(n)\rfloor\right)} \) is \( 1-\frac{1}{p^{2}} \). \newline\newline
Heuristically, we expect that the density of natural numbers \( n \) \newline
satisfying \( \gcd{\left(n, \lfloor P(n)\rfloor\right)} = 1 \) is 
\[
\prod_{p \;\text{prime}} \left(1-\frac{1}{p^{2}}\right).
\]
This product is well-known to evaluate to \( \frac{6}{\pi^2} \), consistent with earlier results for simpler cases. This motivates us to proceed with sieving to establish the result rigorously. However, precise bounds require good control over the error terms, which arise when applying Weyl's equidistribution theorem. To manage these errors effectively, we impose certain Diophantine conditions on the coefficient of \( x \) in \( P \).

\section{Diophantine Approximation}
\begin{df}\label{df} Let $\omega\in\R$ and $B>0$. We say that a number $\alpha\in\R$ is Liouville if we have that $\forall \, n \in \mathbb{N}, \exists$ a pair of integers $(p,q)$ with $q>1$ such that
\begin{equation*}
\left|\alpha-\dfrac{p}{q} \right|\leq q^{-n}.
\end{equation*}
\end{df}

\section{Bounds for Weyl Sums}
Now we obtain bounds on the weyl sums of \quad$\dfrac{mP(dx)}{d}$, \quad independent of $d$ and $m$, for $m$ in a certain range.

\begin{theorem}\label{weyl} Suppose $P \in \mathbb{R}[x]$ with coefficient of $x$, $\alpha$, being \textbf{non-Liouville}. Then we can find constants $\tau,\rho,X_0 > 0$ such that 
\begin{equation*}s_m(X)=:\displaystyle\sum_{x \le X} e \left(m \frac{P(dx)}{d} \right) \le X^{1-\tau} \quad \text{for all}\quad m<X^{\rho} \quad \text{and} \quad X \ge X_0.
\end{equation*}
\end{theorem}

To Prove this we will use the following Theorem.
\begin{theorem}\cite[Theorem 1.6]{Wooley_2012}\label{Wooley} Let $\tau$ and $\delta$ be reals such that $\tau^{-1} \ge 4k(k-1)$ and $\delta>k \tau$. Suppose that $X$ is sufficiently large in terms of $k,\delta$ and $\tau$ and suppose that $|f_{k}(\alpha;X)|>X^{1-\tau}$, then $\exists q,a_1,\cdots a_k \in \mathbb{Z}$ such that \begin{equation*} 1 \le q \le X^{\delta} \quad \text{and} \quad \vert q \alpha_j-a_j\vert \le X^{\delta-j} \;\forall 1 \le j \le k.
\end{equation*}
Here $\alpha:=(\alpha_1,\alpha_2,\cdots,\alpha_k)$ and $f_k(\alpha; x):=\displaystyle\sum_{n \le x} e\left(\alpha_1\,n+\alpha_2\,n^{2}+\cdots \alpha_k\,n^{k}\right)$
\end{theorem}

\begin{proof}[Proof of Theorem 2.1]
Let us consider 
\begin{equation*}\delta < \dfrac{1}{\omega+1},\;\rho=\dfrac{1-(\omega+1)\delta}{2\omega}\quad \text{and}\quad \tau=\dfrac{\delta}{2k(k-1)}
\end{equation*}
Then it is clear that $\delta,\tau$ satisfy the hypothesis for \textbf{Theorem \ref{Wooley}}. Suppose $m<X^{\rho}$ and assume for sake of contradiction $\vert s_m(X)\vert > X^{1-\tau}$. Now, since the coefficient of $x$ in $m \frac{P(dx)}{d}=m\alpha$, we have by \textbf{Theorem \ref{Wooley}} that $\exists q,p \in \mathbb{Z}$, such that
\begin{equation}\label{2}
\left\vert qm\alpha-p\right\vert \le X^{\delta-1} \quad \text{with} \quad gcd(p,q)=1\; \text{and}\; 1 \le q \le X^{\delta}.
\end{equation}
Now since $\alpha$ is non-Liouville, by \textbf{Definition \ref{df}}, we know that there is a natural number $\omega$ and a constant $X_0=X_0(\alpha)$ such that for all $r>X_0, \forall\, s \in \mathbb{N}, \vert r\alpha-s\vert \ge r^{-\omega}$. Now combining this with $(\ref{2})$, we have that
\begin{equation*}
\begin{split}
    (X_0qm)^{-\omega} &\le \vert X_0qm\alpha-pX_0\vert \le X_0X^{\delta-1} \\
    \Longrightarrow m &\ge X_{0}^{-\omega}X^{(1-(\omega+1)\delta)\slash \omega}>X^{(1-(\omega+1)\delta)\slash 2\omega}\\
    \quad &=X^{\rho},\quad \text{contradiction to assumption.}
\end{split}
\end{equation*}
Hence our assumption that $|s_m(X)|>X^{1-\tau}$ is wrong. So we have that \quad $|s_m(X)| \le X^{1-\tau}$\; for all $m<X^{\rho}$ \;as desired.
\qedhere
\end{proof}

\section{Sieving}
\subsection{Controlling Error Terms}
We will now apply the bounds on Weyl sums derived in the previous section along with the Erdos–Turan Discrepancy Theorem to estimate the error terms in the sieving process.

\begin{theorem}Let us define:
\begin{equation*}
\mathcal{A}_{d}(X) =: \{x \le X: d \mid gcd(x,\lfloor P(x) \rfloor\}.
\end{equation*}
Then $\exists$ constants $\mu,C,C_0 > 0$ such that 
\begin{equation*}
\displaystyle\left\vert\mathcal{A}_{d}(X)-\dfrac{X}{d^2}\right\vert \le C_0 \left(\dfrac{X}{d}\right)^{1-\mu}
\end{equation*}
\end{theorem} 

\begin{proof}
By \textbf{Theorem \ref{weyl}}, we know that $\exists \rho,\tau > 0$ such that for large enough $X$ and for all $m<X^{\rho}$, we have that $s_m(X) \le X^{1-\tau}$. Thus we can find a positive constant $C$ such that whenever $\frac{X}{d} > C$, we have that \begin{equation*}
s_m\left(X\slash d\right) \le \left(\dfrac{X}{d}\right)^{1-\tau} \quad \text{for all}\quad m<\left(\dfrac{X}{d}\right)^{\rho}.
\end{equation*}        
Now by Erdos-Turan-Discrepancy Theorem, we know that:
\begin{equation*}
\displaystyle\left\vert\left\vert\mathcal{A}_{d}(X)\right\vert - \frac{X}{d^{2}}\right\vert \le C_0 \left(\dfrac{X}{dT}+\displaystyle\sum_{m \le T} \dfrac{\left\vert s_{m}\left(X\slash d\right)\right\vert}{m}\right).
\end{equation*}
Finally choosing $T=\left(\dfrac{X}{d}\right)^{\rho}$ and $\mu:=\max\{1-\rho, 1-\tau\slash 2\}$, we obtain the desired bound.
\end{proof}
\subsection{Doing The Sieving}
\begin{theorem}Let $P \in \mathbb{R}[x]$ be a polynomial such that the coefficient of $x$ in $P$ is \textbf{non-Liouville}. Then we have that:
\begin{equation*}
\begin{split}
S(X)&=\dfrac{X}{\zeta{\left(2\right)}}+O\left(\dfrac{X}{(\log{\log{X}})^{1\slash 3}}\right)\quad \text{and hence}\quad \lim_{X \rightarrow \infty}\dfrac{S(X)}{X}=\frac{1}{\zeta(2)}\\ 
\text{where} \quad S(X)&=: \#\{x \le X: \gcd{\left(x,\lfloor P(x) \rfloor\right)}=1\}.
\end{split}    
\end{equation*}
\end{theorem}

\begin{proof} Let us define
\begin{equation*}S(X,z):=\#\{x \le X: \gcd{\left(\gcd{\left(x,\lfloor P(x) \rfloor\right)},P_z\right)}=1\}.
\end{equation*}
Let $\epsilon, A> 0$ be arbitrary. Then we notice that
\begin{equation*}
\begin{split}
    S(X)-S(X,z) &\le \displaystyle\sum_{(d,P_{z})=1;\,d \le X} \vert\mathcal{A}_{d}\vert\\
    \quad &= \displaystyle\sum_{(d,P_{z})=1;\,d \le X\slash \log^{\epsilon}{z}} \vert \mathcal{A}_{d}\vert+\displaystyle\sum_{(d,P_{z})=1;\,d \ge X\slash \log^{\epsilon}{z}} \vert\mathcal{A}_{d}\vert
\end{split}
\end{equation*}
Now we shall estimate the the sums on the RHS seperately.
\begin{equation*}
\begin{split}
    \displaystyle\sum_{(d,P_{z})=1;\,d \le X\slash \log^{\epsilon}{z}} \vert \mathcal{A}_{d}\vert &\le \displaystyle\sum_{(d,P_{z})=1;\,d \le X\slash \log^{\epsilon}{z}} \dfrac{X}{d^{2}}+O\left(\left(\dfrac{X}{d}\right)^{1-\mu}\right)\\
    \quad &\le \displaystyle\sum_{d>z}\dfrac{X}{d^{2}}+O\left(\dfrac{X}{\log^{\epsilon \mu}{z}}\right)\\
    \quad &=O\left(\dfrac{X}{\log^{\epsilon \mu}{z}}\right)
\end{split}
\end{equation*}
Furthermore, we also have
\begin{equation*}
\begin{split}
    \displaystyle\sum_{(d,P_{z})=1;\,d \ge X\slash \log^{\epsilon}{z}} \vert\mathcal{A}_{d}\vert &\le {\log^{\epsilon}{z}}\displaystyle\sum_{(d,P_{z})=1;\,d \ge X\slash \log^{\epsilon}{z}} 1\\
    \quad &=\log^{\epsilon}{z}\left(X\left(\displaystyle\prod_{p<z}\left(1-\dfrac{1}{p}\right)+O\left(\dfrac{1}{\log^{A}{X}}\right)\right)\right) \\
    (\text{for} \quad z :&= X^{c \slash \log{log{X}}} \quad\text{where} \quad c=1\slash(2(A+1)))\\
    \quad &\le \dfrac{X}{\log^{1-\epsilon}{z}}+O\left(\dfrac{X}{\log^{A}{X}}\right)
\end{split}
\end{equation*}
Thus in conclusion, we have that
\begin{equation*}
\begin{split}
    S(X)-S(X,z) &=O\left(\dfrac{X}{\log^{1-\epsilon}{z}}\right)+O\left(\dfrac{X}{\log^{\epsilon\mu}{z}}\right)+O\left(\dfrac{X}{\log^{A}{X}}\right)\\
    \text{where}\quad z &= X^{c \slash \log{\log{X}}}.
\end{split}
\end{equation*}
So now it just remains for us to estimate $S(X,z)$.
For this, we observe that
\begin{equation}\label{f}
\begin{split}
    S(X,z) &= \sum_{d \mid P_{z};\, d<X} \mu(d)\vert \mathcal{A}_{d}\vert\\
    \quad &= X\left(\displaystyle\prod_{p<z}\left(1-\dfrac{1}{p^{2}}\right)+O\left(\dfrac{1}{X}\right)\right)+O\left(\displaystyle\sum_{d<X\slash \log^{\epsilon}{z}}\left(\dfrac{X}{d}\right)^{1-\mu}\right)\\
    \quad &+ \displaystyle\sum_{d \mid P_z;\,X\slash \log^{\epsilon}{z}<d \le X}\vert\mathcal{A}_{d}\vert 
\end{split}
\end{equation}
Now if $d \mid P_z$ with $w(d)<\log{\log{d}}+(\log{\log{d}})^{2\slash3}$ and\; $d>\dfrac{X}{\log{z}}$,\; then we can find a prime factor of $d$ of size atleast $M=d^{1\slash{2\log{\log{d}}}}$.
We claim $M>z$. To see this, we note that
\begin{equation*}
\begin{split}
    \log{M} &=\dfrac{\log{d}}{2\log{\log{d}}}\\
    \quad &\ge\dfrac{\log{X}}{\log{\log{X}}}\left(\dfrac{1}{2}-\dfrac{c}{2\log{\log{X}}}\right)\\
    \quad &>\dfrac{\log{X}}{2(A+1)\log{\log{X}}}>\log{z} \quad (\text{for large enough}\quad X)
    \end{split}
\end{equation*}
Thus $M >z.$ But then this means that $d$ has a prime factor $>z$ and thus $d \not\vert P_z$, a contradiction.
Therefore the contribution of the last sum in $\left(\ref{f}\right)$ is atmost
\begin{equation*}
\begin{split}
\#\{n<X: \vert \omega(n)&-\log{\log{n}}\vert>{(\log{\log{n})}}^{2\slash3}\}<\dfrac{X}{(\log{\log{X})^{1\slash 3}}}\\
\quad &\, (\text{Hardy-Ramanujan Theorem})
\end{split}
\end{equation*}
So in conclusion we have that
\begin{equation*}
S(X,z)=\dfrac{X}{\zeta{\left(2\right)}}+O\left(\dfrac{X}{\log^{\epsilon\mu}{z}}\right)+O\left(\dfrac{X}{(\log{\log{X})}^{1\slash 3}}\right)
\end{equation*}
Combining this with the estimate for $S(X)-S(X,z)$ we obtained earlier gives us 
\begin{equation*}
    S(X)=\dfrac{X}{\zeta{\left(2\right)}}+O\left(\dfrac{X}{(\log{\log{X})}^{1\slash 3}}\right) \quad \text{as desired.}
\end{equation*}
\qedhere
\end{proof}

\end{document}